\documentclass[12pt,twoside]{amsart}
\usepackage{amssymb,amsmath,amsthm, amscd, enumerate, mathrsfs, calligra}
\usepackage[all]{xy}
\usepackage{comment}

\title[Arithmetic degrees over function fields]
{Arithmetic degrees for dynamical systems over function fields of characteristic zero}
\author{Yohsuke Matsuzawa}
\author{Kaoru Sano}
\author{Takahiro Shibata} 
\date{}
\address{Graduate school of Mathematical Sciences, the University of Tokyo, Komaba, Tokyo,
153-8914, Japan}
\address{Department of Mathematics, Graduate School of Science, 
Kyoto University, Kyoto 606-8502, Japan}
\address{Department of Mathematics, Graduate School of Science, 
Kyoto University, Kyoto 606-8502, Japan}
\email{myohsuke@ms.u-tokyo.ac.jp}
\email{ksano@math.kyoto-u.ac.jp}
\email{tshibata@math.kyoto-u.ac.jp}

\DeclareMathOperator{\Exc}{Exc}
\DeclareMathOperator{\Supp}{Supp}

\DeclareMathOperator{\Pic}{Pic}
\DeclareMathOperator{\Spec}{Spec}

\DeclareMathOperator{\codim}{codim}
\DeclareMathOperator{\id}{id}

\DeclareMathOperator{\sHom}{\mathscr{H}\kern -.3pt \mathit{om}}

\DeclareMathOperator{\Mor}{Mor}
\DeclareMathOperator{\pr}{pr}
\DeclareMathOperator{\Img}{Im}

\DeclareMathOperator{\Eff}{Eff}
\DeclareMathOperator{\Aut}{Aut}

\newtheorem{thm}{Theorem}[section]
\newtheorem{lem}[thm]{Lemma}
\newtheorem{cor}[thm]{Corollary}
\newtheorem{prop}[thm]{Proposition}
\newtheorem{conj}[thm]{Conjecture}

\newtheorem{prob}[thm]{Problem}

\theoremstyle{definition}
\newtheorem{defn}[thm]{Definition}
\newtheorem{rem}[thm]{Remark}
\newtheorem{ex}[thm]{Example}
\newtheorem*{ack}{Acknowledgments} 

\newtheorem*{notation}{Notation}

\begin{document}

\begin{abstract}
We study arithmetic degree of a dominant rational self-map on a smooth projective 
variety over a function field of characteristic zero.
We see that the notion of arithmetic degree and some related problems over function fields 
are interpreted into geometric ones.
We give another proof of the theorem that the arithmetic degree at any point 
is smaller than or equal to the dynamical degree.
We give a sufficient condition for an arithmetic degree to coincide with the dynamical degree,
and prove that any self-map has so many points whose arithmetic degrees are equal to 
the dynamical degree.
We study dominant rational self-maps on projective spaces in detail. 
\end{abstract}

\maketitle

\tableofcontents 

\section{Introduction}\label{sec1}
Let $K$ be a field where heights $h_X$ on smooth projective varieties $X$ 
can be defined (e.g.~a number field or a function field). 
Given a dominant rational self-map $f: X \dashrightarrow X$ 
on a smooth projective variety $X$ and a rational point $P \in X(K)$, 
it is important to study how the height $h_X(f^m(P))$ varies as $m$ grows. 
As a quantity representing the growth rate of $h_X(f^m(P))$, 
Kawaguchi and Silverman defined arithmetic degree of $f$ at $P$: 
$$\alpha_f(P)=\lim_{m \to \infty} \max \{ h_X(f^m(P)), 1 \}^{1/m}$$
(cf.~\cite{KaSi2} or Definition \ref{def3.2} (ii)). 
On the other hand, there is another invariant for self-maps, 
the (first) dynamical degree of $f$: 
$$\delta_f = \lim_{m \to \infty} ((f^m)^*H \cdot H^{\dim X-1})^{1/m},$$
where $H$ is an ample Cartier divisor on $X$ 
(cf.~Definition \ref{def3.2} (i)).  
It seems important and interesting to investigate relations between those two quantities. 
Especially, it is natural to ask when $\alpha_f(P)$ coincides with $\delta_f$.
Kawaguchi and Silverman conjectured some properties of arithmetic degree 
and a sufficient condition for a rational point $P$ 
on which the arithmetic degree at $P$  
coincides with the dynamical degree (cf.~\cite[Conjecture 6]{KaSi2}). 
These conjectures have been verified in some special cases. 
For details, see Conjecture \ref{conj3.4}, Remark \ref{rem3.5} and Remark \ref{rem3.7}. 

In this article, we study arithmetic degrees over function fields in one variable over an 
algebraically closed field of characteristic zero. 
The advantage of considering height theory over function fields is that 
we can interpret the height of a rational point 
into the degree of the curve corresponding to the rational point. 
A variety $X$ over a function field $K$ of a curve $C$ can be seen as 
the generic fiber of a fibration $\pi: \mathcal X \to C$ over $C$, and then 
a $K$-rational point $P$ of $X$ corresponds to a section $\sigma$ 
of $\pi$. 
In this situation, the height $h_X(P)$ of $P$ is equal to 
$\deg(\sigma^*\mathcal H)$, 
where $\mathcal H$ is a $\pi$-ample Cartier divisor on $\mathcal X$. 
So arithmetic degree is also described by the degrees of divisors. 
We will use this geometric interpretation to deduce the results in this article. 

A fundamental relation on those two types of degree is the following inequality: 
$$\alpha_f(P) \leq \delta_f.$$
This inequality was proved over any field 
where height functions can be defined by Kawaguchi--Silverman and Matsuzawa 
(cf.~{\cite[Theorem 4]{KaSi2}}~and~{\cite{Mat}}).
We will prove it over function fields of characteristic 0 by using the geometric 
interpretation of height (Theorem \ref{thm4}).
Our proof works only over function fields, but  it seems simple and short.

After we obtain the fundamental inequality, it is natural to ask 
when the arithmetic degree of a given rational point attains the dynamical degree.
Over $\overline{\mathbb Q}$, Kawaguchi and Silverman predicts a sufficient condition 
for a rational point to have the arithmetic degree which is equal to the dynamical degree 
(Conjecture \ref{conj3.4} (iv)).
However, over $\overline{k(t)}$, this conjecture does not hold in general 
(Example \ref{ex3.4.2}).
We give a sufficient condition for the arithmetic degree of a rational point 
to attain the dynamical degree as a geometric condition of the corresponding 
section (Theorem \ref{thm_suff}).
For a dynamical system on a projective space,
we will give some other sufficient conditions 
(Theorem \ref{thm5.0.1} and Theorem \ref{thm5.2}).

It is also important to investigate whether there exists a 
rational point whose arithmetic degree attains the dynamical degree.
Over $\overline{k(t)}$ with $k$ an uncountable algebraically closed field of 
characteristic 0, we will prove the result that 
there are densely many points having pairwise disjoint orbits 
such that the arithmetic degree of  them attain the dynamical degree 
(Theorem \ref{thm_orbit}).
Over number fields, this theorem has been proved only for some particular cases 
(cf.~\cite[Theorem 3]{KaSi1} and \cite[Theorem 1.7]{MSS}).
More strongly,
for a self-map on a smooth projective rational variety,
we can take such points over a fixed function field (Theorem \ref{thm5.4}).

\begin{notation}
\begin{itemize}
\item[]
\item Throughout this article, $k$ denotes an algebraically closed field of characteristic zero,
and $\overline{k(t)}$ denotes the algebraic closure of the rational function field of one 
variable over $k$.
\item For a rational map $f : X \dashrightarrow Y$, $I_f$ denotes the indeterminacy locus 
of $f$.
\item A \textit{curve} simply means a smooth projective variety 
of dimension 1 unless otherwise stated.
\item For any $\mathbb R$-valued function $h(x)$, we set 
$h^+(x) = \max \{ h(x), 1 \}$.
\item Let $f$, $g$ and $h$ be  $\mathbb R$-valued functions on a domain $S$.
The equality $f = g + O(h)$ means that there is a positive constant $C$ such that 
$|f(x)-g(x)| \leq C |h(x)|$ for every $x \in S$.
The equality $f=g + O(1)$ means that there is a positive constant $C'$ such that 
$|f(x)-g(x)| \leq C'$ for every $x \in S$.
\end{itemize}
\end{notation}

\begin{ack}
The authors would like to thank Professors Osamu Fujino and Shu Kawaguchi 
for reading a manuscript of this paper and giving useful comments.
The first author wold like to thank Professor Tomohide Terasoma for attending his seminar and
giving valuable comments.
\end{ack}

%%%%%%%%%%%%%%%%%%%%%%%%%%%%%%%%%%%%%%%%%%%%%%%%%%%%%%%%%%%%

\section{Height functions for varieties over function fields}\label{sec2}

In this section, we define the (Weil) height functions on projective varieties 
over $\overline{k(t)}$, and see that there is another description of height 
in terms of the degree of a divisor on a curve. 
Basic facts of (Weil) height functions over function fields is explained 
for example in \cite[Chapter 3, \S 3]{Lan} and \cite[B.10]{HiSi}.  
So we omit most of the proofs. 

First, we define the height functions on projective spaces over function fields. 

\begin{defn}\label{def2.1}
Let $C$ be a (smooth projective) curve over $k$ and $p \in C$ a closed point. 
We define a valuation $v_p: K(C)^\times \to \mathbb Z$ as 
$$v_p(f) = (\mathrm{the\ multiplicity\ of\ }f \mathrm{\ at\ }p).$$
This is nothing but the discrete valuation associated to the discrete valuation ring 
$\mathcal O_{C, p}$. 
Set $v_p(0)=\infty$. 
\end{defn}

\begin{defn}\label{def2.2}
Let $C$ be a curve over $k$.  
Take $P \in \mathbb P^n(K(C))$.
Represent $P$ by homogeneous coordinates as 
$P=(f_0: f_1: \cdots : f_n)$, where $f_i \in K(C)$. 
We define the \textit{height function on $\mathbb P^n(K(C))$ relative to $K(C)$} 
as 
$$h_{K(C)}(P)= \sum_{p \in C} -\min \{v_p(f_0), \ldots, v_p(f_n)\}.$$ 
\end{defn}

\begin{rem}\label{rem2.3}
(i) Since $\# \{p \in C | v_p(f) \neq 0 \}$ is finite for any $f \in K(C)^\times$, 
$\sum_{p \in C} -\min \{v_p(f_0), \ldots, v_p(f_n)\}$ is in fact a finite sum. 

(ii) $h_{K(C)}(P)$ is independent of the representation of $P$. 
For, if we represent $P$ as $P= (gf_0 : \cdots : gf_n)$, where $g \in K(C)^\times$, then 
\begin{align*}
\sum_p -\min_i v_p(gf_i) &= \sum_p -\min_i(v_p(g)+v_p(f_i)) \\
&= \sum_p -v_p(g)+ \sum_p -\min v_p(f_i) \\
&= \sum_p -\min_i v_p(f_i),
\end{align*}
since $\sum_p -v_p(g) = - \deg ( (g) ) = 0$. 
\end{rem}

Next, we give the definition of the (absolute) height function 
on $\mathbb P^n(\overline{k(t)})$ which is 
essentially compatible with the previous definition of height functions over function fields.

\begin{lem}[cf.~{\cite[Lemma B.2.1]{HiSi}}]\label{lem2.5}
Let $\phi: C' \to C$ be a finite surjective morphism of two curves. 
Set $K=K(C), K'=K(C')$. 
Take $P \in \mathbb P^n(K)$. 
Then 
$$ \frac{1}{[K: k(t)]} h_K(P)= \frac{1}{[K': k(t)]} h_{K'}(P).$$
\end{lem}

We define the (absolute) height function on $\mathbb P^n(\overline{k(t)})$ as follows: 

\begin{defn}\label{def2.6}
Take $P \in \mathbb P^n(\overline{k(t)})$. 
We define the \textit{height function on $\mathbb P^n(\overline{k(t)})$} 
as 
$$h(P)= \frac{1}{[K: k(t)]} h_K(P),$$
where $K$ is a finite extension field of $k(t)$ such that 
$P \in \mathbb P^n(K)$. 
$h(P)$ is independent of the choice of $K$ by Lemma \ref{lem2.5}. 
\end{defn}

Using the height functions on projective spaces, we define the height function 
on a projective variety $X$ associated to a Cartier divisor. 
We prepare the following lemma.

\begin{lem}[cf.~{\cite[Proposition B.2.4 and Theorem B.3.1]{HiSi}}]\label{lem2.8}
Let $X$ be a projective variety over $\overline{k(t)}$ and $D$ a base point free 
Cartier divisor on $X$. 
\begin{itemize}
\item[(i)] 
Let $\phi: X \to \mathbb P^n$, $\psi: X \to \mathbb P^n$ be two morphisms associated to 
the complete linear system $|D|$. 
Then $h \circ \phi = h \circ \psi + O(1)$. 
\item[(ii)] 
Let $\phi_1: X \to \mathbb P^{n_1}$, $\phi_2: X \to \mathbb P^{n_2}$ be morphisms 
associated to very ample Cartier divisors $A_1, A_2$ on $X$, respectively. 
Let $\phi: X \to \mathbb P^n$ be a morphism associated to $A_1 + A_2$. 
Then $h \circ \phi = h \circ \phi_1 + h \circ \phi_2+ O(1)$. 
\end{itemize}
\end{lem}

Using Lemma \ref{lem2.8}, we can define the height function 
associated to a Cartier divisor up to the difference of a bounded function.

\begin{defn}\label{def2.9}
Let $X$ be a projective variety over $\overline{k(t)}$. 

(i) Let $A$ be a base point free Cartier divisor on $X$. 
We define a \textit{height function on $X$ associated to $A$} as 
$$h_A = h \circ \phi_A,$$ 
where $\phi_A$ is a morphism associated to $|A|$. 
By Lemma \ref{lem2.8} (i), $h_A$ is well-defined up to a bounded function. 

(ii) Let $D$ be a Cartier divisor on $X$.  
We define a \textit{height function on $X$ associated to $D$} as 
$$h_D= h \circ \phi_A - h \circ \phi_B,$$ 
where $A, B$ are base point free Cartier divisors such that $D \sim A-B$ and 
$\phi_A, \phi_B$ are morphisms associated to $|A|, |B|$, respectively.  

Note that we can always take such $A$ and  $B$. 
Assume that we have other base point free Cartier divisors 
$A', B'$ such that $D \sim A' - B'$. 
Then $A+B' \sim A'+B$. 
Lemma \ref{lem2.8} (i) implies that 
$h_{A+B'}=h_{A'+B} + O(1)$. 
Moreover, by Lemma \ref{lem2.8} (ii), 
$h_{A+B'} 
=h_A+ h_{B'}+ O(1)$  
and $h_{A'+B}=h_{A'}+h_B+O(1)$. 
So $h_A - h_B= h_{A'}- h_{B'}+O(1)$. 
Hence $h_D$ is well-defined up to a bounded function.  
\end{defn}

\begin{prop}[cf.~{\cite[Theorem B.3.2]{HiSi}}]\label{prop2.9.1}
Let $X$ be a projective variety over $\overline{k(t)}$. 
\begin{itemize}
\item[(i)]
Let $D, D'$ be Cartier divisors on $X$ and $n, n'$ be integers. Then 
$$h_{nD+n'D'}= nh_D + n'h_{D'}+O(1).$$ 
\item[(ii)]
Let $f: X \to Y$ be a morphism to a projective variety $Y$ and $E$ a Cartier divisor 
on $Y$. 
Then 
$$h_{f^*E}=h_E \circ f+ O(1).$$ 
\end{itemize}
\end{prop}

Next, we give another description of height in terms of the degree of a divisor 
on a curve.  

\begin{lem}\label{lem2.10}
Let $X, S$ be integral schemes and $f: X \to S$ a morphism of finite type. 
Let $\eta$ be the generic point of $S$ and $X_\eta$ the generic fiber of $f$. 
For a rational section $\sigma: S  \dashrightarrow X$ of $f$ 
(i.e.~a rational map $\sigma: S  \dashrightarrow X$ such that 
$f \circ \sigma=\mathrm{id}_S$), 
$P \in X_\eta(k(\eta))$ denotes the rational point obtained by the base change 
of $\sigma$ along $k(\eta) \to S$. 
Then the mapping $\sigma \mapsto P$ gives a one-to-one correspondence
between the set of rational sections of $f$ and $X_\eta(k(\eta))$. 
\end{lem}

This lemma follows from an elementary scheme-theoretic argument. 
When $X$ is a projective variety and $S=C$ is a (smooth projective) 
curve, Lemma \ref{lem2.10} is reduced to the following.

\begin{prop}\label{prop2.11}
Let $C$ be a curve over $k$ and set $K=K(C)$. 
\begin{itemize}
\item[(i)]
Let $\pi: X \to C$ be a surjective morphism from a projective variety $X$ to $C$ 
and $X_\eta$ the generic fiber of $\pi$.
Then $X_\eta(K)$ corresponds one-to-one to the set of sections of $\pi$. 
\item[(ii)]
Let $Y_k$ be a projective variety over $k$ and set $Y_K=Y_k \times_k K$.
Then $Y_K(K)$ corresponds one-to-one to the set of $k$-morphisms from $C$ to $Y_k$. 
\end{itemize}
\end{prop}

\begin{proof}
(i) Note that any rational map $\varphi: C \dashrightarrow X$ is in fact a morphism since 
$\codim I_\varphi \geq 2$. 
So the assertion follows from Lemma \ref{lem2.10}. 

(ii) Apply (i) to the projection $\pr_C: Y_k \times_k C \to C$. 
Note that the sections of $\pr_C$ correspond one-to-one to the $k$-morphisms 
from $C$ to $Y_k$. 
\end{proof}

Here is another description of height by the degree of divisors for projective spaces. 

\begin{prop}[cf.~{\cite[Lemma B.10.1]{HiSi}}]\label{prop2.12}
Let $C$ be a curve and set $K=K(C)$. 
Take $P \in \mathbb P^n(K)$ and let $g: C \to \mathbb P^n_k$ be 
the corresponding morphism. 
Then 
$$h(P)= \frac{1}{[K:k(t)]}\deg(g^*\mathcal O(1)).$$
\end{prop}

\begin{proof}
For $P=(f_0: \cdots: f_n) \in \mathbb P^n(K)$, $g$ is represented as 
$g(x)=(f_0(x): \cdots: f_n(x))$.
Let $H_i=(x_i)$ be the hyperplane of $\mathbb P^n_k$ associated to the $i$-th coordinate. 
Assume that $g(C) \subset H_i$ for some $i$. 
Then $f_i=0$, so 
$$h_K(P)= \sum_{p \in C} -\min_{j \neq i} v_p(f_j).$$
On the other hand, $\deg(g^*\mathcal O_{\mathbb P^n}(1)) = 
\deg(g^*\mathcal O_{H_i}(1))$. 
Therefore we can replace $\mathbb P^n_k$ by $H_i \cong \mathbb P^{n-1}_k$. 
If $g$ is a constant mapping, then $f_i \in k$ for all $i$ and 
$h_K(P)=\deg(g^*\mathcal O(1))=0$. 
So we may assume that $g(C)$ is not contained $H_i$ for all $i$. 
 
Set $U=\Spec A = C \setminus g^{-1}(H_0)$ and $f_i=(g|_U)^*(x_i/x_0) \in A$. 
Then $P=(1: f_1: \cdots: f_n)$. So
\begin{align*}
h_K(P) &= \sum_{p \in C} -\min \{v_p(1), v_p(f_1), \ldots, v_p(f_n) \} \\
&= \sum_{p \in C \setminus U} -\min \{v_p(1), v_p(f_1), \ldots, v_p(f_n) \} \\
&= \sum_{p \in C \setminus U} -\min_{1 \leq i \leq n} v_p(f_i) \\
&= \sum_{p \in C \setminus U} -\min_{1 \leq i \leq n} v_p(g^*(x_i/x_0)) \\
&= \sum_{p \in C \setminus U} -\min_{1 \leq i \leq n} v_p(g^*H_i)
+ \sum_{p \in C \setminus U} v_p(g^*H_0), \\
\end{align*}
where the second equality holds because $v_p(f_i) \geq 0$ for $p \in U$ 
and the third equality holds because $\displaystyle \min_{1 \leq i \leq n} v_p(f_i) <0$ 
for $p \in C \setminus U$. 
Obviously 
$$\sum_{p \in C \setminus U} v_p(g^*H_0) = 
\sum_{p \in C} v_p(g^*H_0) = \deg(g^*H_0).$$ 
For $p \in C \setminus U$, $v_p(g^*H_i) >0$ is equivalent to $g(p) \in H_i$. 
Since $g(p) \in H_0$, $g(p) \not\in H_i$ for some $i>0$ and then $v_p(g^*H_i)=0$. 
Hence $\displaystyle \min_{1 \leq i \leq n}v_p(g^*H_i)=0$ for all $p \in C \setminus U$. 
As a consequence, 
$$h(P)=\frac{1}{[K:k(t)]} \deg(g^*H_0) = \frac{1}{[K:k(t)]} \deg(g^*\mathcal O(1)).$$ 
\end{proof}

By Proposition \ref{prop2.12}, we can see that $h(P) \geq 0$ 
for any $P \in \mathbb P^n(K)$. 
Furthermore, 
for a rational point $P \in \mathbb P^n(K)$ corresponding to a morphism 
$g: C \to \mathbb P_k^n$, $P \in \mathbb P^n(k)$ if and only if $g$ is a constant map. 
So we obtain the following.

\begin{prop}\label{prop2.12.1}
\begin{itemize}
\item[ ]
\item[(i)]
$h(P) \geq 0$ for any $P \in \mathbb P^n(\overline{k(t)})$. 
\item[(ii)]
For $P \in \mathbb P^n(\overline{k(t)})$, 
$h(P) = 0$ if and only if $P \in \mathbb P^n(k)$. 
\end{itemize}
\end{prop}

We give a description of height by the degree of divisors for a projective variety 
over $\overline{k(t)}$ which has a model over a curve or, more strongly, over $k$. 

\begin{defn}\label{defn2.12.2}
Let $X$ be a projective variety over $\overline{k(t)}$ and $H$ 
an ample Cartier divisor on $X$.
We define a function $\tilde h_H : X(\overline{k(t)}) \to \mathbb R_{\geq 0}$ as follows.
Fix  a model 
$(X_C \overset{\pi}{\to} C, H_C)$ of $(X, H)$ over a curve $C$, that is, 
a projective variety $X_C$ over $k$ with a surjection $\pi:X_C \to C$ whose 
geometric generic fiber is $X$, and a $\pi$-ample Cartier divisor $H_C$ on $X_C$ 
such that 
$(X \to X_C)^*H_C \sim H$.
For any $P \in X(\overline{k(t)})$, take a curve $C_1$ with $K(C_1) \supset K(C)$ and 
the section $\sigma_1$ of $\pi_{C_1}: X_C \times_C C_1 \to C_1$ corresponding to $P$,
and set $H_{C_1}=(X_C \times_C C_1 \to X_C)^*H_C$ and
$$\tilde h_H(P)= \frac{1}{[K(C_1):k(t)]}\deg(\sigma_1^*H_{C_1}).$$
\end{defn}

\begin{prop}\label{prop2.12.3}
Notation is as in Definition \ref{defn2.12.2}.
Then $\tilde h_H$ is a well-defined height function associated to $H$.
\end{prop}

\begin{proof}
Take any point $P \in X(\overline{k(t)})$.
Take curves $C_i$ with $K(C_i)\supset K(C)$ and 
the sections $\sigma_i$ of $\pi_{C_i}: X_{C_i}=X_C \times_C C_i \to C_i$ for $i=1,2$.
To see the well-definedness of $\tilde h_H$,
we may assume that $K(C_2) \supset K(C_1)$.
$$
\xymatrix{
C_2 \ar[d]^{\sigma_2} \ar[r]^{\phi_2} & C_1 \ar[d]^{\sigma_1} \\ 
X_{C_2} \ar[d]^{\pi_2} \ar[r]^{\psi_2}   & X_{C_1} \ar[d]^{\pi_1} \ar[r]^{\psi_1}  
& X_C \ar[d]^{\pi} \\
C_2 \ar[r]^{\phi_2} & C_1 \ar[r]^{\phi_1} & C
}
$$
Set $H_{C_1}=\psi_1^*H_C$ and $H_{C_2}=\psi_2^*H_{C_1}$.
Then 
\begin{align*}
\frac{\deg(\sigma_2^*H_{C_2})}{[K(C_2):k(t)]}
&= \frac{\deg(\phi_2^*\sigma_1^*H_{C_1})}{[K(C_2):K(C_1)][K(C_1):k(t)]}\\
&= \frac{\deg(\sigma_1^*H_{C_1})}{[K(C_1):k(t)]}.
\end{align*}
So it follows that $\tilde h_H(P)$ is well-defined.

Take a sufficiently large integer $N$ such that there is a morphism 
$\iota_C: X_C \to \mathbb P^n_k \times C$ over $C$ with 
$\iota_C^*\mathcal O_{\mathbb P^n_C}(1) \sim_\pi NH_C$.
Take a Cartier divisor $D_C$ on $C$ such that 
$\iota_C^*\mathcal O_{\mathbb P^n_C}(1) \sim NH_C + \pi^*D_C$.
Let $\iota:X \to \mathbb P^n_{\overline{k(t)}}$ be the base change of $\iota_C$ 
by $\Spec \overline{k(t)} \to C$.
Then the function $\frac{1}{N}h \circ \iota$ is a height function associated to $H$.
$\iota(P) \in \mathbb P^n(\overline{k(t)})$ corresponds to the morphism 
$\pr_{\mathbb P^n_k} \circ \iota_{C_1} \circ \sigma_1: C_1 \to \mathbb P^n_k$,
where $\iota_{C_1}$ be the base change of $\iota_C$ by $C_1 \to C$.
We compute 
\begin{align*}
\frac{1}{N}h(\iota(P))
&= \frac{\deg((\pr_{\mathbb P^n_k} \circ \iota_{C_1} \circ \sigma_1)^*
\mathcal O_{\mathbb P^n_k} (1))}{N[K(C_1):k(t)]}
 \ \ \ \mathrm{(by\ Proposition\ \ref{prop2.12})} \\
&= \frac{\deg(\sigma_1^* (NH_{C_1}+\psi_1^*\pi^*D_C))}{N[K(C_1):k(t)]}  \\
&= \frac{\deg(\sigma_1^*(NH_{C_1}))+[K(C_1):K(C)]\deg(D_C)}{N[K(C_1):k(t)]} \\
&= \tilde h_H(P)
+ \frac{\deg(D_C)}{N[K(C):k(t)]}.
\end{align*}
So $\tilde h_H$ is a height function associated to $H$.
\end{proof}

%%%%%%%%%%%%%%%%%%%%%%%%%%%%%%%%%%%%%%%%%%%%%%%%%%%%%%%%%%%%%%%%%%%%%%%

\section{Arithmetic degrees for dynamical systems over function fields}\label{sec3}

\begin{defn}\label{def3.1}
A \textit{dynamical system over a field $K$} is a pair $(X, f)$ of
a smooth projective variety $X$ over $K$ and 
a dominant rational self-map $f: X \dashrightarrow X$ over $K$. 

Let $(X, f)$ be a dynamical system over a field $K$. 
Set 
$$X_f = \{ P \in X(K) | f^m(P) \not\in I_f 
\mathrm{\ for\ every\ }m \geq 0\}.$$
For $P \in X_f$, set 
$$O_f(P) = \{ f^m(P) | m \in \mathbb Z_{\geq 0} \},$$ 
which we call the (\textit{forward}) \textit{$f$-orbit of $P$}. 

Let $H$ be an ample divisor on $X$.
The (\textit{first}) \textit{dynamical degree of $f$} is the number 
$$\delta_f = \lim_{m \to \infty}((f^m)^*H \cdot H^{\dim X-1})^{1/m} \ \in [1, \infty).$$
It is known that the limit converges and $\delta_f$ is independent of 
the choice of $H$. 
\end{defn}

\begin{defn}\label{def3.2}
Let $(X,f)$ be a dynamical system over an algebraically closed field $K$ where 
heights are well-defined (e.g.~$\overline{k(t)}$ or $\overline{\mathbb Q}$).
Take an ample Cartier divisor $H$ on $X$ 
and a rational point $P \in X_f$. 
The \textit{arithmetic degree of $f$ at $P$} is defined as 
$$\alpha_f(P) = \lim_{m \to \infty} h_H^+(f^m(P))^{1/m},$$
where $h_H$ is a height function associated to $H$ and 
$h_H^+(P)= \max \{ h_H(P), 1 \}$.
Note that we do not know whether the limit converges (cf.~Conjecture \ref{conj3.4} (i)).
Similarly, $\overline \alpha_f(P), \underline \alpha_f(P)$ are defined as
$$\overline \alpha_f(P) = \limsup_{m \to \infty} h_H^+(f^m(P))^{1/m},$$
$$\underline \alpha_f(P) = \liminf_{m \to \infty} h_H^+(f^m(P))^{1/m}.$$
\end{defn}

\begin{prop}[cf.~{\cite[Proposition 14]{KaSi2}}]\label{prop3.3}
In the notation of Definition \ref{def3.2}, 
$\alpha_f(P)$, $\overline \alpha_f(P)$ and $\underline \alpha_f(P)$ 
are independent of the choices of $H$ and $h_H$. 
\end{prop}

\begin{proof}
It is obvious that $\alpha_f(P)$, $\overline \alpha_f(P)$ and  $\underline \alpha_f(P)$ are 
independent of 
the choice of $h_H$ for a fixed ample Cartier divisor $H$. 

Let $H'$ be another ample Cartier divisor on $X$. 
Take a sufficiently large integer $N$ such that $NH-H'$ is ample. 
Then 
\begin{align*}
\lim_{m \to \infty} h_H^+(f^m(P))^{1/m} 
&= \lim_{m \to \infty} (\frac{1}{N}h_{NH}^+(f^m(P)))^{1/m} \\ 
&= \lim_{m \to \infty} h_{NH}^+(f^m(P))^{1/m} \\
&= \lim_{m \to \infty} \max \{ h_{NH-H'}(f^m(P))+ h_{H'}(f^m(P)), 1 \}^{1/m} \\
&\geq \lim_{m \to \infty} h_{H'}^+(f^m(P))^{1/m}, 
\end{align*}
where we take $h_{NH-H'}$ as a non-negative function. 
Similarly, 
$$\lim_{m \to \infty} h_{H'}^+(f^m(P))^{1/m} 
\geq \lim_{m \to \infty} h_H^+(f^m(P))^{1/m}.$$ 
So $\alpha_f(P)$ is independent of $H$. 
The proofs for $\overline \alpha_f(P)$ and $\underline \alpha_f(P)$ are similar. 
\end{proof}

Now we introduce the following conjecture (over $\overline{\mathbb Q}$) 
given by Kawaguchi and Silverman (see \cite{KaSi2}).

\begin{conj}[The Kawaguchi--Silverman conjecture over $\overline{\mathbb Q}$]\label{conj3.4}
Let $(X, f)$ be a dynamical system over $\overline{\mathbb Q}$. 
\begin{itemize}
\item[(i)] The limit defining $\alpha_f(P)$ exists for every $P \in X_f$. 
\item[(ii)] $\alpha_f(P)$ is an algebraic integer for every $P \in X_f$. 
\item[(iii)] $\{ \alpha_f(P) | P \in X_f \}$ is a finite set. 
\item[(iv)] 
Take $P \in X_f$. 
Assume that 
$O_f(P)$ is Zariski dense in $X$. 
Then 
$\alpha_f(P) = \delta_f$.  
\end{itemize} 
\end{conj}

\begin{rem}\label{rem3.5}
In the case when 
$f$ is a morphism (see \cite[Theorem 2]{KaSi3}),
Kawaguchi and Silverman showed that  (i), (ii), and (iii) of 
Conjecture \ref{conj3.4} hold even over global fields of characteristic zero.
\end{rem}

\begin{rem}\label{rem3.7}
It is known that Conjecture \ref{conj3.4} holds in several cases 
(cf.~\cite[Remark 1.3]{MSS}).
\end{rem}

As the number field case, we consider:

\begin{prob}\label{prob3.4.1}
Let $(X, f)$ be a dynamical system over $\overline{k(t)}$. 
Take a point $P \in X_f$. 
When the equality $\alpha_f(P)=\delta_f$ holds?
\end{prob}

The following examples show that the function field case of Conjecture \ref{conj3.4} (iv) 
is not true.

\begin{ex}\label{ex3.4.2}
(i) Let $f: \mathbb P_{\overline{k(t)}}^1 \to \mathbb P_{\overline{k(t)}}^1$ be a surjective 
endomorphism with $\delta_f >1$.
Take a $k$-valued non-preperiodic point $P \in \mathbb P^1(k)$.
Then $O_f(P)$ is Zariski dense in $\mathbb P_{\overline{k(t)}}^1$,
but $\alpha_f(P)=1 < \delta_f$.

(ii) Define $f: \mathbb A_{\overline{k(t)}}^2 \to \mathbb A_{\overline{k(t)}}^2$ 
as $f(x,y)=(x^2,y^3)$.
Then $f$ naturally extends to the morphism 
$f: \mathbb P_{\overline{k(t)}}^2 \to \mathbb P_{\overline{k(t)}}^2$ 
and $\delta_f=3$.
Take a point $P=(t,2) \in \mathbb A^2(k(t))$.
Then $f^m(P)=(t^{2^m}, 2^{3^m})$ and 
$$\alpha_f(P)=\lim_{m \to \infty} \max \{\deg(t^{2^m}), \deg(2^{3^m}) \}^{1/m}
= \lim_{m \to \infty} (2^m)^{1/m} = 2.$$

We show that $O_f(P)=\{(t^{2^m}, 2^{3^m})\}_{m=0}^\infty$ is dense in 
$\mathbb P_{\overline{k(t)}}^2$.
It is enough to show that $O_f(P)$ is dense in $\mathbb A_{k(t)}^2$.
Suppose $O_f(P)$ is contained in the zero locus of a polynomial 
$\phi(t,x,y) \in k(t)[x,y]$.
Multiplying $\phi$ with a polynomial in $k[t]$, we may assume that 
$\phi \in k[t,x,y]$.
Set $\phi(t,x,y)=\phi_r(t,y)x^r + \phi_{r-1}x^{r-1}+ \cdots + \phi_0(t,y)$, $\phi_r(t,y) \neq 0$.
By assumption, $\phi(t,t^{2^m},2^{3^m})=0$ as a polynomial in $k[t]$.
Since 
$$\deg(t^{2^mr}) 
> \deg(\phi(t,t^{2^m},2^{3^m})-\phi_r(t,2^{3^m})t^{2^mr})
= \deg(-\phi_r(t,2^{3^m})t^{2^mr})$$ 
for sufficiently large $m$,
it follows that $\phi_r(t,2^{3^m})t^{2^mr}=0$ as a polynomial in $k[t]$ for sufficiently large 
$m$.
Therefore $\phi_r(t,y)=0$ as a polynomial in $k[t,y]$, which is a contradiction.
So $O_f(P)$ is Zariski dense in $\mathbb P_{\overline{k(t)}}^2$.
\end{ex}

In the rest of this article, we will find some other conditions for an arithmetic degree 
to coincide with the dynamical degree.

\section{A fundamental inequality}\label{sec4}

There is a fundamental inequality between arithmetic degrees and dynamical degrees: 

\begin{thm}[{\cite[Theorem 4]{KaSi2}}~and~{\cite{Mat}}]\label{thm4}
$K$ denotes an algebraically closed field where heights are well-defined. 
Let $(X, f)$ be a dynamical system over $K$.   
Then $$\overline \alpha_f(P) \leq \delta_f$$ 
holds for any $P \in X_f$. 
\end{thm}

\begin{rem}\label{rem4.0}
This inequality was stated  by Kawaguchi and Silverman (see \cite[Theorem 4]{KaSi2}),  
and a correct proof of it was given by Matsuzawa (see \cite{Mat}).
\end{rem}

We will give another proof of the inequality over $\overline{k(t)}$.

\begin{thm}\label{thm4.4}
Let $(X, f)$ be a dynamical system over $\overline{k(t)}$.  
Then the inequality 
$$\overline \alpha_f(P) \leq \delta_f$$ 
holds for any $P \in X_f$. 
\end{thm}

To prove Theorem \ref{thm4.4}, we prepare some lemmas.
To begin with,
we define a model of a dynamical system over $\overline{k(t)}$.

\begin{defn}\label{def3.1.1}
Let $(X, f)$ be a dynamical system over $\overline{k(t)}$. 
A \textit{model of the dynamical system $(X, f)$ 
over a curve $C$} is a pair $(X_C \overset{\pi}{\to} C, f)$ of 
a surjective morphism $\pi: X_C \to C$ from a smooth projective $k$-variety $X_C$ 
to $C$ 
and a dominant rational self-map $f_C: X_C \dashrightarrow X_C$ over $C$ 
such that $X_C \times_C \overline{k(t)} = X$ and 
the base change of $f_C$ along $\Spec \overline{k(t)} \to C$ is equal to $f$. 
\end{defn}

\begin{lem}\label{lem4.1}
Let $(X, f)$ be a dynamical system over $\overline{k(t)}$. 
Then there exists a model of $(X, f)$ over a curve $C$. 
\end{lem}

Such a model is obtained by resolution of singularities.

\begin{lem}[cf.~{\cite[Proposition 19]{KaSi2}}]\label{lem4.1.2}
Let $f: X \dashrightarrow Y$ and $g: Y \dashrightarrow Z$ be dominant rational maps of 
smooth projective varieties. 
Take a Cartier divisor $H$ on $Z$ and a curve $C$ on $X$. 
\begin{itemize}
\item[(i)] 
If $C \not\subset I_f$, $f(C) \not\subset I_g$ and $H$ is nef, then 
$$((g \circ f)^*H \cdot C) \leq (f^*g^*H \cdot C).$$ 
\item[(ii)]
If $C \cap I_f = \varnothing$ and $f(C) \cap I_g = \varnothing$, then 
$$((g \circ f)^*H \cdot C) = (f^*g^*H \cdot C).$$
\end{itemize}
\end{lem}

\begin{proof}
Since a nef divisor is the limit of a sequence of ample divisors, 
we may assume that $H$ is ample. 
$$
\xymatrix{
X'' \ar[d]_{\mu'} \ar[dr]^{\tilde h} \\ 
X' \ar[d]_\mu \ar@{-->}[r]^h  \ar[dr]^{\tilde f} & Y' \ar[d]^\nu \ar[dr]^{\tilde g} \\
X \ar@{-->}[r]^f & Y \ar@{-->}[r]^g & Z
}
$$
In the above diagram, 
$\tilde f: X' \to Y$ (resp. $\tilde g: Y' \to Z$) is an elimination of indeterminacy of $f$ 
(resp.~$g$) by blowing up smooth centers in $I_f$ (resp.~$I_g$), 
$h = \nu^{-1} \circ f \circ \mu$, 
and $\tilde h: X'' \to Y'$ is an elimination of indeterminacy of $h$ 
by blowing up smooth centers in $I_h$. 
Then 
\begin{align*}
&(g \circ f)^*H \cdot C \leq f^*g^*H \cdot C\ \ \cdots \mathrm{(1)}\\
&\iff (\mu \circ \mu')_* (\tilde g \circ \tilde h)^*H \cdot C 
\leq \mu_* \tilde f^* \nu_* \tilde g^*H \cdot C \\
&\iff \mu_* \mu'_* \tilde h^* \tilde g^* H \cdot C 
\leq \mu_* \tilde f^* \nu_* \tilde g^*H \cdot C.
\end{align*}
Here 
$\mu_* \tilde f^* \nu_* \tilde g^*H  
=\mu_* \mu'_* \mu'^*\tilde f^* \nu_* \tilde g^*H 
=\mu_* \mu'_* \tilde h^* \nu^* \nu_* \tilde g^*H$. 
Set 
$$E = \nu^* \nu_* \tilde g^* H- \tilde g^*H,$$ 
then 
(1) is equivalent to the inequality 
$$\mu_* \mu'_* \tilde h^* E \cdot C \geq 0.$$ 
By negativity lemma (cf.~\cite[Lemma 3.39]{KoMo}), 
$E$ is an effective and $\nu$-exceptional divisor. 
Take a curve $C'$ on $X'$ such that $\mu(C')=C$ and a curve $C''$ on $X''$ 
such that $\mu'(C'')=C'$. 

$\nu(\tilde h(C''))= \tilde f(\mu'(C''))= \tilde f(C')= f(C) \not\subset I_g$, 
so $\tilde h(C'') \not\subset \Exc(\nu)$. 
In particular, $\tilde h(C'') \not\subset \Supp E$ and then 
$C'' \not\subset \Supp \tilde h^*E$. 
Hence 
$$C=\mu(\mu'(C'')) \not\subset \mu(\mu'(\Supp \tilde h^*E)) 
= \Supp \mu_* \mu'_* \tilde h^*E.$$ 
This implies (1). 

(ii) is obvious since $f|_C: C \to f(C)$ and $g|_{f(C)}: f(C) \to Y$ are morphisms. 
\end{proof}

The following lemma is a variant of Lemma \ref{lem4.1.2}.

\begin{lem}\label{lem4.2}
Let $f: X \dashrightarrow Y$ be a dominant rational map of smooth projective varieties 
and $g: C \to X$ a morphism from a curve $C$. 
Take a Cartier divisor $H$ on $Y$. 
\begin{itemize}
\item[(i)]
If $g(C) \not\subset I_f$ and $H$ is nef, then 
$$\deg((f \circ g)^*H) \leq \deg(g^* f^*H).$$
\item[(ii)]
If $g(C) \cap I_f = \varnothing$, then 
$$\deg((f \circ g)^*H) = \deg(g^* f^*H).$$
\end{itemize}
\end{lem}

\begin{proof}
(i) Since a nef divisor is the limit of a sequence of ample divisors, 
we may assume that $H$ is ample. 
$$
\xymatrix{
 & X' \ar[d]^\mu \ar[dr]^{\tilde f} \\
C \ar[ur]^{\tilde g} \ar[r]^g & X \ar@{-->}[r]^f & Y
}
$$
In the above diagram, 
$\tilde f: X' \to Y$ is an elimination of indeterminacy of $f$ 
by blowing up smooth centers in $I_f$, and
we can define the composition $\tilde g = \mu \circ g$ 
by the assumption that $g(C) \not\subset I_f$. 
Moreover it is a morphism. 

We compute 
\begin{align*}
\deg(g^*f^*H) - \deg((f \circ g)^*H)
&= \deg(\tilde g^* \mu^* \mu_* \tilde f^* H - \tilde g^* \tilde f^*H) \\
&= \deg(\tilde g^*E), 
\end{align*}
where we set $E= \mu^*\mu_* \tilde f^*H - \tilde f^*H$. 
By negativity lemma(cf.~\cite[Lemma 3.39]{KoMo}), 
$E$ is an effective and $\mu$-exceptional divisor on $X'$. 
Moreover $\tilde g(C) \not\subset \Supp E$ 
since $\mu(\tilde g(C)) = g(C) \not\subset I_f$ and 
$\mu(E) \subset I_f$. 
So $\deg(\tilde g^*E) \geq 0$. 

(ii) is obvious since both $g: C \to g(C)$ and $f|_{g(C)}: g(C) \to Y$ are morphisms. 
\end{proof}

\begin{lem}\label{lem4.3}
Let $X$ be a smooth projective variety  
with an ample Cartier divisor $H$.  
$\overline{\Eff}(X) \subset N^1(X)_{\mathbb R}$ denotes the pseudo-effective cone of $X$. 
Take a 1-cycle $Z \in N_1(X)_{\mathbb R}$. Then there is a constant $M>0$ such that 
$$(E \cdot Z) \leq M(E \cdot H^{\dim X-1})$$
holds for any $E \in \overline{\Eff}(X)$. 
\end{lem}

\begin{proof}
Note that $(E \cdot H^{\dim X-1})>0$ for any $E \in \overline{\Eff}(X) \setminus \{0\}$ 
(see \cite[Lemma 20]{KaSi2}). 
We define a function $f: \overline{\Eff}(X) \setminus \{0\} \to \mathbb R$ as 
$$f(E)=\frac{(E \cdot Z)}{(E \cdot H^{\dim X-1})}.$$
Take a norm $||\cdot||$ on $N^1(X)_{\mathbb R}$ 
and set $S=\{E \in  \overline{\Eff}(X) |\  ||E||=1 \}$. 
Then we can take an upper bound $M>0$ of $f|_S$ since $S$ is compact. 
But $f$ satisfies $f(cE)=f(E)$ for $E \in N^1(X)_{\mathbb R}$ and $c>0$, so $M$ is in fact 
an upper bound of $f$. This implies the claim. 
\end{proof}

\begin{lem}\label{lem4.3.1}
Let $(X, f)$ be a dynamical system over $\overline{k(t)}$ with a 
model $(X_C \overset{\pi}{\to} C, f_C)$ over a curve $C$ or $k$. 
Then $\delta_f=\delta_{f_C}$. 
\end{lem}

\begin{proof}
We define the $k$-th dynamical degree and the $k$-th relative dynamical degree: 
$$\lambda_k(f_C)= \lim_{m \to \infty} ((f_C^m)^* H_C^k \cdot H_C^{n+1-k})^{1/m},$$
$$\lambda_k(f_C|\pi)=\lim_{m \to \infty} ((f_C^m)^*H_C^k \cdot H_C^{n-k} \cdot F)^{1/m}.$$
Note that $\lambda_1(f_C)=\delta_{f_C}$.

Set $n= \dim X$.
Take an ample divisor $H_C$ on $X_C$ and a general fiber $F$ of $\pi$.
Fix an integer $m>0$.
Take an elimination of indeterminacy of $f_C^m$: 
\[
\xymatrix{
&\Gamma_C \ar[ld]_{p_C} \ar[rd]^{g_C}&\\
X_C \ar@{-->}[rr]_{f_C^m}&&X_C
}
\] 
Pulling it back along $\overline{k(t)} \to C$,  we get the following diagram:
\[
\xymatrix{
X \ar[r]^{\eta}&X_C\\
\Gamma \ar[u]^p \ar[d]_g \ar[r]^{\eta_{\Gamma}}&\Gamma_C \ar[u]_{p_C} \ar[d]^{g_C}\\
X \ar[r]_{\eta}&X_C
}
\]
Set $H=\eta^*H_C$.
We can show that $g^*\eta^*=\eta_\Gamma^* g_C^*$ and 
$p_*\eta_\Gamma^*=\eta^*p_{C*}$.
So we have 
$$(f^m)^*H = p_*g^*\eta^*H_C = p_*\eta_\Gamma^* g_C^*H_C 
= \eta^* p_{C*} g_C^* H_C = \eta^*(f_C^m)^*H_C.$$
So $((f^m)^*H \cdot H^{n-1})= (\eta^*(f_C^m)^*H_C \cdot (\eta^*H_C)^{n-1})$.
Hence $((f^m)^*H \cdot H^{n-1})$ is equal to 
the coefficient of the monomial $t_1 \cdots t_n$ for the numerical polynomial 
$$\chi(X, t_1 \eta^*(f_C^m)^*H_C + t_2 \eta^*H_C + \cdots + t_n \eta^*H_C).$$
For any Cartier divisor $D$ on $X_C$ and a general fiber $F$ of $\pi$, 
the equality $\chi(X, \eta^*D)=\chi(F, D|_F)$ holds.
So
$$\chi(X, t_1 \eta^*(f_C^m)^*H_C + t_2 \eta^*H_C + \cdots + t_n \eta^*H_C)$$
$$= \chi(F, t_1 (f_C^m)^*H_C|_F + t_2 H_C|_F + \cdots + t_n H_C|_F).$$
Hence we have 
$(\eta^*(f_C^m)^*H_C \cdot (\eta^*H_C)^{n-1}) 
= ((f_C^m)^*H_C|_F \cdot (H_C|_F)^{n-1}) 
= ((f_C^m)^*H_C \cdot H_C^{n-1} \cdot F)$, and so 
$\lambda_1(f)= \lambda_1(f_C|\pi)$.

On the other hand, by \cite[Theorem 1.4]{Tru}, 
\begin{align*}
\lambda_1(f_C)
&=\max \{ \lambda_1 (f_C|\pi)\lambda_0(\id_{C}), 
\lambda_0 (f_C|\pi) \lambda_1(\id_{C}) \} \\
&= \max \{ \lambda_1(f_C|\pi) , 1 \} \\
&=\lambda_{1}(f_C|\pi).
\end{align*}
Note that $\lambda_q(\mathrm{id}_C)=1$ for all $q$ and $\lambda_0(f_C|\pi)=1$ by definition.
So $$\delta_f =\lambda_1(f)= \lambda_1(f_C|\pi)=\lambda_1(f_C)=\delta_{f_C}.$$
\end{proof}

\begin{proof}[Proof of Theorem \ref{thm4.4}]
Take $P \in X_f$. 
Put $n = \dim X$. 
By Lemma \ref{lem4.1}, 
we can take a model $(X_C \overset{\pi}{\to} C, f_C)$ over a curve $C$. 
We may assume that 
$P$ corresponds to a section $\sigma: C \to X_C$ of $\pi$. 

Take an ample Cartier divisor $H_C$ on $X_C$ and set $H=(X \to X_C)^*H_C$.
By Lemma \ref{lem4.3.1}, 
$$\delta_f= \delta_{f_C} = \lim_{m \to \infty} ((f_C^m)^*H_C \cdot H_C^n)^{1/m}.$$ 
On the other hand, by Proposition \ref{prop2.12.3}, 
\begin{align*}
\overline \alpha_f(P) 
&= \limsup_{m \to \infty} \tilde h^+_H(f^m(P))^{1/m} \\
&= \limsup_{m \to \infty}  \deg^+((f_C^m \circ \sigma)^*H_C)^{1/m}.  
\end{align*}
Note that $\Img(\sigma) \not\subset I_{f_C^m}$ since $P \not\in I_{f^m}$. 
By Lemma \ref{lem4.2} (i), 
$$\deg((f_C^m \circ \sigma)^*H_C) 
\leq \deg(\sigma^*(f_C^m)^*H_C) = ((f_C^m)^*H_C \cdot \sigma_*C).$$
It is obvious that $(f_C^m)^*H_C \in \overline{\Eff}(X_C)$ for every $m$. 
So, by Lemma \ref{lem4.3}, there is a constant $M>0$ such that the inequality 
$$((f_C^m)^*H_C \cdot \sigma_*C) \leq M((f_C^m)^*H_C \cdot H_C^n)$$
holds for every $m$. 
Therefore we have 
\begin{align*}
\overline \alpha_f(P) 
&\leq \limsup_{m \to \infty}  ((f_C^m)^*H_C \cdot \sigma_*C)^{1/m} \\
&\leq \limsup_{m \to \infty} (M((f_C^m)^*H_C \cdot H_C^n))^{1/m} \\
&= \limsup_{m \to \infty} ((f_C^m)^*H_C \cdot H_C^n)^{1/m} \\
&= \delta_{f_C} \\
&= \delta_f. 
\end{align*}
\end{proof}

%%%%%%%%%%%%%%%%%%%%%%%%%%%%%%%%%%%%%%%%%%%%%%%%%%%%%%%%%%%%%%

\section{A sufficient condition}\label{sec_suff}

Let $(X,f)$ be a dynamical system over $\overline{k(t)}$.
In this section, we give a sufficient condition of a rational point $P \in X_f$
whose arithmetic degree attains the dynamical degree.

\begin{thm}\label{thm_suff}
Let $(X,f)$ be a dynamical system over $\overline{k(t)}$ and
$(X_C \overset{\pi}{\to} C, f_C)$ a model over a curve $C$.
Take a rational point $P \in X_f$ corresponding to a section $\sigma:C \to X_C$ of $\pi$.
Assume that 
\begin{itemize}
\item
$\sigma(C) \cap I_{f_C^m} = \varnothing$ for every $m \geq 1$ and 
\item
$(E \cdot \sigma(C)) >0$ for any $E \in \overline{\Eff}(X) \setminus \{ 0\}$.
\end{itemize}
Then $\alpha_f(P)$ exists and $\alpha_f(P)=\delta_f$.
\end{thm}

We prepare the following lemma.

\begin{lem}\label{lem_norm}
Let $X$ be a smooth projective variety and $Z \subset X$ a 1-cycle such that 
$(E \cdot Z) >0$ for any $E \in \overline{\Eff}(X) \setminus \{0\}$.
We define a non-negative function $||\cdot ||_Z: N^1(X)_{\mathbb R} \to \mathbb R$
as 
$$||v||_Z = \inf \{ (v_1 \cdot Z)+ (v_2 \cdot Z) | v=v_1 - v_2, \ v_1, v_2\ 
\mathrm{are\ effective\ classes} \}.$$
\begin{itemize}
\item[(i)]
$||v||_Z=(v \cdot Z)$ for any effective class $v \in N^1(X)_\mathbb R$.
\item[(ii)]
$||\cdot ||_Z$ is a norm on $N^1(X)_{\mathbb R}$.
\end{itemize}
\end{lem}

\begin{proof}
(i) For effective classes $v_1, v_2$ such that $v=v_1 - v_2$, 
we have $(v \cdot Z)= (v_1 \cdot Z) - (v_2 \cdot Z) \leq (v_1 \cdot Z) + (v_2 \cdot Z)$.
So $||v||_Z=(v \cdot Z)$.

(ii) It is easy to see that 
\begin{itemize}
\item
$||cv||_Z=|c| \cdot ||v||_Z$ for any $c \in \mathbb R$ and $v \in N^1(X)_\mathbb R$ and 
\item
$||v+w||_Z \leq ||v||_Z + ||w||_Z$ for any $v, w \in N^1(X)_\mathbb R$.
\end{itemize}
Take $v \in N^1(X)_\mathbb R$ and assume that $||v||_Z=0$.
Then we have 
$$\{ v_n^+ \}_n, \{ v_n^- \}_n \subset N^1(X)_\mathbb R$$ 
such that 
$v=v_n^+ - v_n^-$ for every $n$ 
and 
$$\lim_{n \to \infty} ((v_n^+ \cdot Z)+(v_n^- \cdot Z))=0.$$
So $\lim_{n \to \infty} (v_n^{\pm} \cdot Z) =0$.
Since $(w \cdot Z)>0$ for any $w \in \overline{\Eff}(X) \setminus \{0\}$,
it follows that $\lim_{n \to \infty} v_n^{\pm}=0$.
Therefore $v=\lim_{n \to \infty}(v_n^+ - v_n^-)=0$.
So $||\cdot||_Z$ satisfies the conditions of norm.
\end{proof}

\begin{proof}[Proof of Theorem \ref{thm_suff}]
Set $n=\dim X$.
We have 
\begin{align*}
\delta_f
&= \delta_{f_C}\ \ \ \mathrm{(by\ Lemma\ \ref{lem4.3.1})} \\
&= \lim_{m \to \infty} ((f_C^m)^*H_C \cdot H_C^n)^{1/m} \\
&= \lim_{m \to \infty} ||(f_C^m)^*H_C||_{H_C^n}^{1/m}.\ \ \ 
\mathrm{(by\ Lemma\ \ref{lem_norm}\ (i))}
\end{align*}
Note that $||\cdot ||_{H_C^n}$ is a norm since $(E \cdot H_C^n) >0$ for every 
$E \in \overline{\Eff}(X) \setminus \{0\}$ (cf.~\cite[Lemma 20]{KaSi2}).
We obtain  
\begin{align*}
\delta_f
&=\lim_{m \to \infty} ||(f_C^m)^*H_C||_{\sigma(C)}^{1/m}\ \ \ 
\mathrm{(since\ }||\cdot ||_{H_C^n}\ \mathrm{is\ equivalent\ to\ }||\cdot ||_{\sigma(C)})\\
&=\lim_{m \to \infty} ((f_C^m)^*H_C \cdot \sigma(C))^{1/m} \ \ \ 
\mathrm{(by\ Lemma\ \ref{lem_norm}\ (i))} \\
&=\lim_{m \to \infty} \deg^+(\sigma^*(f_C^m)^*H_C)^{1/m} \\
&=\lim_{m \to \infty} \deg^+((f_C^m \circ \sigma)^*H_C)^{1/m}\ \ \ 
\mathrm{(by\ Lemma\ \ref{lem4.2}\ (ii))} \\
&= \lim_{m \to \infty} \tilde h_H^+(f^m(P))^{1/m}\ \ \ 
\mathrm{(by\ Proposition\ \ref{prop2.12.3})} \\
&=\alpha_f(P).
\end{align*}
\end{proof}

%%%%%%%%%%%%%%%%%%%%%%%%%%%%%%%%%%%%%%%%%%%%%%%%%%%%%%%%%%%%%

\section{Arithmetic degrees for projective spaces}\label{sec5}

In this section, we study arithmetic degrees for dynamical systems on 
projective spaces. 

At first, we give some sufficient conditions for a rational point at which 
the arithmetic degree attains the dynamical degree. 
\begin{lem}\label{lem5.0}
Let $C$ be a curve. Set $X= \mathbb P_k^n \times C$. 
Take an pseudo-effective Cartier divisor $E$ on $X$ 
and a general fiber $F$ of $\pi= \pr_C$.  
Then $\mathcal O_X(E) \equiv \mathcal O_X(d) \otimes \mathcal O_X(eF)$ 
for some $d, e \in \mathbb Z_{\geq 0}$. 
\end{lem}

\begin{proof}
It is sufficient to prove the claim for effective divisors, so we may assume that 
$E$ is effective.
Since $\Pic(X)$ is generated by $\mathcal O_X(1)$ and $\pi^* \Pic(C)$, 
there are an integer $d$ and a divisor $D_C$ on $C$ such that 
$\mathcal O_X(E) \cong \mathcal O_X(d) \otimes \pi^* \mathcal O_C(D_C)$. 
Set $e=\deg D_C$. 
Then $\mathcal O_X(E) \equiv 
\mathcal O_X(d) \otimes \mathcal O_X(eF)$. 

Since $E|_F$ is effective and
$\mathcal O_F(E|_F) \cong \mathcal O_{\mathbb P_{\mathbb C}^n}(d)$, 
$d \geq 0$. 
By projection formula,  
$$\pi_*(\mathcal O_X(d) \otimes \pi^* \mathcal O_C(D_C)) 
\cong \pi_*(\mathcal O_X(d)) \otimes \mathcal O_C(D_C)
\cong S^d(\mathcal O_C^{\oplus n+1}) \otimes \mathcal O_C(D_C).$$
Since $H^0(C,S^d(\mathcal O_C^{\oplus n+1}) \otimes \mathcal O_C(D_C))= 
 H^0(X, \mathcal O_X(d) \otimes \pi^* \mathcal O_C(D_C)) \neq 0$, 
$D_C$ is effective. 
So $e=\deg D_C \geq 0$. 
\end{proof}

\begin{thm}\label{thm5.0.1}
Let $(X= \mathbb P_{\overline{k(t)}}^n, f)$ be a dynamical system over 
$\overline{k(t)}$ and 
$(X_C= \mathbb P_k^n \times C \overset{\pr_{C}}{\to} C, f_C)$  
a model of $(X, f)$ over a curve $C$. 
Take a morphism $g: C \to \mathbb P^n_k$ corresponding to a rational point 
$P_g \in X_f$ and set 
$\sigma_g= (g, \mathrm{id}_C): C \to X_C$. 
\begin{itemize}
\item[(i)]
Assume that 
$g$ is non-constant and 
$\Img(\sigma_g) \cap I_{f_C^m}=\varnothing$ for every $m \geq 1$. 
Then $\alpha_f(P_g)$ exists and 
$\alpha_f(P_g)= \delta_f$. 
\item[(ii)]
Assume the following conditions.
\begin{itemize}
\item[($*$)] For every $m \geq 0$, 
$\pr_{\mathbb P_k^n} \circ f_C^m \circ \sigma_g$ is non-constant. 
\item[($**$)] There is a sequence of positive integers 
$m_1 < m_2 < \ldots$ such that 
$\Img(f_C^{m_k} \circ \sigma_g) \cap I_{f_C^{m_{k+1}-m_k}}=\varnothing$ 
for every $k \geq 1$ 
and $\lim_{k \to \infty} (m_k/m_{k+1})=0$.  
\item[($***$)] The limit $\alpha_f(P_g)$ exists. 
\end{itemize}
Then the equality $\alpha_f(P_g)=\delta_f$ holds. 
\end{itemize}
\end{thm}

\begin{proof}
(i) Let $F$ be a general fiber of $\pr_C$.
Then 
$$\overline{\Eff}(X_C)
=\mathbb R_{\geq 0} \mathcal O_X(F) + \mathbb R_{\geq 0} \mathcal O_X(1)$$
by Lemma \ref{lem5.0}.
It is obvious that $\sigma_g$ satisfies the assuption of Theorem \ref{thm_suff}.
So (i) follows from Theorem \ref{thm_suff}.

(ii) 
For $m \geq 1$, set $(f_C^m)^*\mathcal O_{X_C}(1) \equiv \mathcal O_{X_C}(d_m) \otimes 
\mathcal O_{X_C}(e_mF)$. 
Then $d_m, e_m \geq 0$ by Lemma \ref{lem5.0}. 
It is clear that $(f^m)^*\mathcal O_X(1) 
\equiv \mathcal O_X(d_m)$, and so 
$$\delta_f= \lim_{m \to \infty} ((f^m)^*\mathcal O_X(1) \cdot 
\mathcal O_X(1)^{n-1})^{1/m}=\lim_{m \to \infty} d_m^{1/m}.$$

Set $b_m= \deg((f_C^m \circ \sigma_g)^*\mathcal O_{X_C}(1))$. 
By ($*$), $b_m \geq 1$ for every $m$. 
So we have 
$$\alpha_f(P_g)=\lim_{m \to \infty} \max \{b_m, 1\}^{1/m} 
= \lim_{m \to \infty} b_m^{1/m}.$$
For $k \geq 1$, set $l_k=m_{k+1}-m_k$. 
We compute
\begin{align*}
b_{m_{k+1}}
&= \deg((f_C^{m_{k+1}} \circ \sigma_g)^*\mathcal O_{X_C}(1)) \\
&= \deg(f_C^{l_k} \circ f_C^{m_k} \circ \sigma_g)^*\mathcal O_{X_C}(1))\\
&= \deg(f_C^{m_k} \circ \sigma_g)^*(f_C^{l_k})^*\mathcal O_{X_C}(1))\ \ \ 
\mathrm{(by\ }(**)\ \mathrm{and\ Lemma\ \ref{lem4.2}\ (ii))}  \\
&= \deg((f_C^{m_k} \circ \sigma_g)^*\mathcal O_{X_C}(d_{l_k}) 
\otimes \mathcal O_{X_C}(e_{l_k}F)) \\
&\geq \deg((f_C^{m_k} \circ \sigma_g)^*\mathcal O_{X_C}(d_{l_k})) \\
&=d_{l_k} b_{m_k}  \\
&\geq d_{l_k}.
\end{align*}
Note that $\lim_{k \to \infty} l_k =\infty$ by the assumption that $\lim_{k \to \infty} (m_k/m_{k+1})=0$. 
Hence 
\begin{align*}
\alpha_f(P_g)
&= \lim_{m \to \infty} (b_m)^\frac{1}{m}\ \ \ 
\mathrm{(by\ Proposition\ \ref{prop2.12.3})} \\
&= \lim_{k \to \infty} (b_{m_{k+1}})^\frac{1}{m_{k+1}} \\
&\geq \lim_{k \to \infty} (d_{l_k})^{\frac{1}{l_k} \cdot (1-\frac{m_k}{m_{k+1}})} \\
&= \delta_f.
\end{align*}
Combining with Theorem \ref{thm4}, it follows that  
$\alpha_f(P_g)=\delta_f$. 
\end{proof}

Next, we show that a sufficiently general morphism $g: C \to \mathbb P_k^n$ 
of a given sufficiently large degree corresponds to a rational point whose arithmetic degree 
attains the dynamical degree.

\begin{defn}\label{def5.1}
Let $C$ be a curve of genus $g(C)$ over $k$ and $d, n$ positive integers. 
$\Mor_d(C, \mathbb P_k^n)$ denotes the set of morphisms $g: C \to \mathbb P_k^n$ 
such that $\deg(g^*\mathcal O(1)) = d$. 
\end{defn}

$\Mor_d(C, \mathbb P_k^n)$ has a structure of $k$-variety with the evaluation 
$e:\Mor_d(C, \mathbb P_k^n) \times C \to \mathbb P_k^n$ 
which maps $(g,p)$ to $g(p)$. 
Moreover, if $\Mor_d(C, \mathbb P_k^n)$ is non-empty, we have 
$$\dim \Mor_d(C, \mathbb P_k^n) \geq (n+1)d+ n(1-g(C))$$ 
(cf.~\cite[1.1]{KoMo}).

\begin{thm}\label{thm5.2}
Let $(X= \mathbb P_{\overline{k(t)}}^n, f)$ be a dynamical system over 
$\overline{k(t)}$ and 
$(X_C= \mathbb P_k \times C \overset{\pr_C}{\to} C, f_C)$  
a model of $(X, f)$ over a curve $C$ of genus $g(C)$. 
Take a positive integer $d$ satisfying 
$d > \frac{n(g(C)-1)}{n+1}.$ 
$P_g \in X(\overline{k(t)})$ denotes the rational point corresponding to 
$g \in \Mor(C, \mathbb P_k^n)$. 
Then $\alpha_f(P_g)$ exists and 
$\alpha_f(P_g) = \delta_f$ 
for a sufficiently general $g \in \Mor_d(C, \mathbb P_k^n)$. 
\end{thm}

\begin{proof}
Let $M \subset \Mor_d(C, \mathbb P_k^n)$ be an irreducible component 
of maximal dimension. 
Then $\dim M>0$ by assumption. 
Set $\Phi =(e, \mathrm{id}_C): M \times C \to X_C$, 
where $e$ is the evaluation. 
For any $g \in M$ and $\rho \in \Aut(\mathbb P_k^n)$,
we have $\deg(g^* \rho^* \mathcal O_{\mathbb P^n}(1)) 
= \deg(g^* \mathcal O_{\mathbb P^n}(1))= d$, so 
$\Aut(\mathbb P_k^n)$ acts on $M$.
Fix $g_0 \in M$.
For any $(x, p) \in X_C$, 
we can take $\rho \in \Aut(\mathbb P_k^n)$ such that $\rho(g_0(p))= x$. 
Then $\Phi(\rho \circ g_0, p)= ((\rho \circ g_0)(p), p)= (x, p)$.  
So it follows that $\Phi$ is surjective. 

For every $m \geq 1$, we compute 
\begin{align*}
\dim \Phi^{-1}(I_{f_C^m})
&\leq (\dim(M \times C) - \dim X_C) + \dim I_{f_C^m} \\
&\leq (\dim M + 1 - \dim X_C) + \dim X_C -2 \\
&= \dim M -1.
\end{align*}
Hence $\pr_M(\Phi^{-1}(I_{f_C^m})) \subset M$ is a proper subset of $M$ 
for every $m \geq 1$.

For $g \in M$, $\sigma_g=(g, \mathrm{id}_C): C \to X_C$ 
denotes the corresponding section of $\pr_C$.
For $g \in M$, we have 
\begin{align*}
\sigma_g(C) \cap I_{f_C^m} = \varnothing 
&\iff \Phi(\{g\} \times C) \cap I_{f_C^m} = \varnothing \\
&\iff  \{g\} \times C \cap \Phi^{-1}(I_{f_C^m}) = \varnothing \\
&\iff g \not\in \pr_M(\Phi^{-1}(I_{f_C^m})). 
\end{align*}
Set $(f_C^m)^*\mathcal O_{X_C}(1) \equiv \mathcal O_{X_C}(d_m) \otimes 
\mathcal O_{X_C}(e_mF)$, where $F$ is a general fiber of $\pr_C$. 
Then $d_m, e_m \geq 0$ by Lemma \ref{lem5.0}. 
Take $g \in M \setminus \bigcup_{m \geq 1} \pr_M(\Phi^{-1}(I_{f_C^m}))$.
We compute
\begin{align*}
\underline \alpha_f(P_g)
&= \liminf_{m \to \infty} \deg ((f_C^m \circ \sigma_g)^*\mathcal O_{X_C}(1))_+^{1/m} \ \ \ 
\mathrm{(by\ Proposition\ \ref{prop2.12.3})}\\ 
&= \liminf_{m \to \infty} \deg (\sigma_g^*(f_C^m)^*\mathcal O_{X_C}(1))_+^{1/m}\ \ \ 
\mathrm{(by\ Lemma\ \ref{lem4.2}\ (ii))} \\
&= \liminf_{m \to \infty} (\mathcal O_{X_C}(d_m) \otimes 
\mathcal O_{X_C}(e_mF) \cdot \sigma_{g*}C)_+^{1/m} \\
&\geq \liminf_{m \to \infty} (\mathcal O_{X_C}(d_m) \cdot \sigma_{g*} C)_+^{1/m} \\
&= \liminf_{m \to \infty} (d_m (\mathcal O_{X_C}(1) \cdot \sigma_{g*}C))_+^{1/m} \\
&= \liminf_{m \to \infty} d_m^{1/m} \\
&= \delta_f.
\end{align*}
Note that $(\mathcal O_{X_C}(1) \cdot \sigma_{g*}C)
= (\mathcal O_{\mathbb P_k^n}(1) \cdot g_*C) >0$ since $d= \deg(g) >0$. 
Combining with Theorem \ref{thm4.4}, it follows that $\alpha_f(P_g)$ exists and 
$\alpha_f(P_g)=\delta_f$. 
\end{proof}

%%%%%%%%%%%%%%%%%%%%%%%%%%%%%%%%%%%%%%%%%%%%%%%%%%%%%%%%%%%%%%%%%%

\section{Construction of orbits}\label{sec_orbit}

In this section, we consider a problem  
on the existence of the rational points at which the arithmetic degree attains 
the dynamical degree.

Over $\overline{\mathbb Q}$, the following problem is studied in some papers 
(cf. \cite[Theorem 3]{KaSi1} and \cite[Theorem 1.7]{MSS}).

\begin{prob}\label{prob5.3}
Let $(X, f)$ be a dynamical system over $\overline{\mathbb Q}$. 
Is there a subset 
$S \subset X_f$ such that 
\begin{itemize}
\item 
$\alpha_f(P)$ exists and $\alpha_f(P)=\delta_f$ for every $P \in S$, 
\item
$O_f(P) \cap O_f(Q) = \varnothing$ 
if $P, Q \in S$ and $P \neq Q$, and 
\item 
$S$ is a Zariski dense subset of $X$.
\end{itemize}  
or not? 
\end{prob}

We give an affirmative answer for any dynamical system 
over $\overline{k(t)}$,
where $k$ is an uncountable algebraically closed field of characteristic 0.

\begin{thm}\label{thm_orbit}
Assume that $k$ is an uncountable algebraically closed field of characteristic 0.
Let $(X, f)$ be a dynamical system over 
$\overline{k(t)}$. 
Then there exists 
a subset $S \subset X_f$ such that 
\begin{itemize}
\item 
$\alpha_f(P)$ exists and $\alpha_f(P)=\delta_f$ for every $P \in S$, 
\item
$O_f(P) \cap O_f(Q) = \varnothing$ 
if $P, Q \in S$ and $P \neq Q$, and 
\item 
$S$ is a Zariski dense subset of $X$.
\end{itemize}  
\end{thm}

\begin{lem}\label{lem_points}
Let $k$ be an uncountable algebraically closed field and $X$ an algebraic scheme 
of positive dimension over $k$.
Let $Z_1, Z_2, \ldots \subset X$ be proper closed subsets of $X$.
Then $\bigcup_i Z_i \neq X$ and 
there exists a countable set $M$ 
of $k$-valued points of $X \setminus \bigcup_i Z_i$ such that 
$M$ is Zariski dense in $X$.
\end{lem}

\begin{proof}
Replacing $X$ by an affine open subset, we may assume that $X$ is affine.
By Noether's normalization lemma, there is a finite cover $\phi: X \to \mathbb A_k^n$.
Replacing $X$ and $Z_1, Z_2, \ldots$ by $\mathbb A_k^n$ and $\phi(Z_1), \phi(Z_2), \ldots$,
we may assume that $X=\mathbb A_k^n$.

We prove the claim by induction on $n$.
Assume that $n=1$.
Then $Z_i(k)$ is a finite set for every $i$ and 
$\mathbb A_k^1(k)=k$ is uncountable,
so $\bigcup_i Z_i \neq \mathbb A_k^1$.
We take an infinite subset $M$ of $\mathbb A_k^1(k) \setminus \bigcup_i Z_i(k)$.
Then $M$ is a Zariski dense subset of $\mathbb A_k^1$.

Assume that the claim holds for $\mathbb A_k^1, \mathbb A_k^2,\ldots, \mathbb A_k^{n-1}$.
Define $p: \mathbb A_k^n \to \mathbb A_k^{n-1}$ 
and $q: \mathbb A_k^n \to \mathbb A_k^1$ as 
$p(x_1, \ldots ,x_n)=(x_1,\ldots ,x_{n-1})$ and $q(x_1, \ldots ,x_n)=x_n$.
Let $\{ Z_j'\}_j$ (resp.~$\{ Z_k''\}_k$) be the members of $\{ Z_i \}_i$ such that 
$p(Z_j') \neq \mathbb A_k^{n-1}$ (resp.~$p(Z_k'') = \mathbb A_k^{n-1}$).
Let $W_k \subset \mathbb A_k^{n-1}$ be the set of points $w \in \mathbb A_k^{n-1}$ 
such that the fiber $(p|_{Z_k''})^{-1}(w)= p^{-1}(w) \cap Z_k''$ of $p|_{Z_k''}$ over $w$ 
has positive dimension.
Then $W_k$ is a proper closed subset of $\mathbb A_k^{n-1}$.
By induction hypothesis, $\bigcup_j \phi(Z_j') \cup \bigcup_k W_k \neq \mathbb A_k^{n-1}$ 
and we can take a countable subset 
$M' \subset \mathbb A_k^{n-1}(k) \setminus (\bigcup_j \phi(Z_j')(k) \cup \bigcup_k W_k(k))$
such that $M'=\{ a_m\}_{m=1}^\infty$ is Zariski dense in $\mathbb A_k^{n-1}$.
For every $m$ and $k$, $p^{-1}(a_m) \cap Z_k'' \neq \mathbb A_k^1$ 
since $a_m \not\in W_k$.
So $\bigcup_{m,k} (p^{-1}(a_m) \cap Z_k'') \not\subset \mathbb A_k^1$ 
and we can take a countable subset 
$M'' \subset \mathbb A_k^1(k) \setminus \bigcup_{m,k} (p^{-1}(a_m)(k) \cap Z_k''(k))$ 
such that $M''$ is Zariski dense in $\mathbb A_k^1$, by induction hypothesis.
Set $M=M' \times M'' \subset \mathbb A_k^{n-1} \times \mathbb A_k^1$.
Then it is clear that $M$ satisfies the claim.
\end{proof}

\begin{proof}[Proof of Theorem \ref{thm_orbit}]
Take a model $(X_{C_0} \overset{\pi_{C_0}}{\to} C_0, f_{C_0})$  
of $(X, f)$ over a curve $C_0$. 
For any curve $C$ with a finite morphism $C \to C_0$,
$(X_C \overset{\pi_C}{\to} C, f_C)$ denotes the pull-back of 
$(X_{C_0} \overset{\pi_{C_0}}{\to} C_0, f_{C_0})$ by $C \to C_0$ 
and $\psi_C: X_C \to X_{C_0}$ denote the projection.
For a section $\sigma:C \to X$ and a finite morphism $C' \to C$ of curves,
$(\sigma)_{C'}: C' \to X \times_C C'$ denotes the pull-back of $\sigma$ by $C' \to C$.

By Lemma \ref{lem_points}, we can take a countable subset 
$M=\{a_i\}_{i=1}^\infty \subset X_{C_0}$ 
such that
\begin{itemize}
\item 
$M$ is Zariski dense in $X_{C_0}$ and 
\item 
$a_i \not\in I_{f_{C_0}^m}$ for every $m \geq 1$ and 
$i \geq 1$.  
\end{itemize}

We will construct rational points $P_1, P_2, \ldots \in X$ inductively.
Let $C_k \to C_{k-1} \to \cdots \to C_1 \to C_0$ be a sequence of finite morphisms of 
curves and 
$P_i \in X$ a rational point corresponding to a section 
$\sigma_i: C_i \to X_{C_i}$ of $\pi_{C_i}$ 
for each $1 \leq i \leq k$.
Assume that $P_1, \ldots, P_k \in X$ satisfy the following condition $(*)_k$: 
\begin{itemize}
\item
$P_i \in X_f$ for $1 \leq i \leq k$,
\item
$\alpha_f(P_i)=\delta_f$ for $1 \leq i \leq k$, 
\item
$O_f(P_i) \cap O_f(P_j) = \varnothing$ if $1 \leq i, j \leq k$ and $i \neq j$, and
\item
$a_i \in \Img(\psi_{C_i} \circ \sigma_i)$ for $1 \leq i \leq k$.
\end{itemize}
Set $n=\dim X$.
Note that $X_{C_k}$ is smooth outside a finite union of fibers of $\pi_{C_k}$.

Let $p_k: X_k \to X_{C_k}$ be a resolution of $(X_{C_k})_{\mathrm{red}}$ 
whose exceptional locus is contained in a finite union of fibers of $\pi_{C_k}$.
By blowing up a point in $(p_k \circ \psi_{C_k})^{-1}(a_{k+1})$,
we may assume that $(p_k \circ \psi_{C_k})^{-1}(a_{k+1})$ has codimension 1.
We take a very ample divisor $H$ on $X_k$ and 
suitable members $H_1, \ldots, H_n \in |H|$, 
and set $C_{k+1}=H_1 \cap \cdots  \cap H_n$.
Let $\iota: C_{k+1} \to X_k$ denote the inclusion.
We can choose $H_1, \ldots, H_n$ as satisfying 
\begin{itemize}
\item[(I)]
$C_{k+1}$ is a smooth and irreducible curve satisfying 
$\Img(\pi_{C_k} \circ p_k \circ \iota)=C_k$, 
\item[(II)]
$C_{k+1} \not\subset p_k^{-1} (f_{C_k}^m)^{-1}(I_{f_{C_k}})$ for every $m \geq 0$,
\item[(III)] 
$C_{k+1} \cap I_{f_k^m} = \varnothing$ for every $m \geq 1$, 
\item[(IV)] 
$C_{k+1} \not\subset (f_k^{m'})^{-1}(\Img(f_k^m \circ p_k^{-1} \circ (\sigma_i)_{C_k}))$ 
for every $m, m'$ and $1 \leq i \leq k$, and 
\item[(V)]
$a_{k+1} \in \Img(\psi_{C_k} \circ p_k \circ \iota)$.
\end{itemize}
Set $\phi=\pi_{C_k} \circ p_k \circ \iota: C_{k+1} \to C_k$.
Then we obtain the following diagram: 
$$
\xymatrix{
C_{k+1} \ar@(d, ul)[ddr]_{\mathrm{id}} \ar@(r, ul)[drr]^{p_k \circ \iota} \ar[dr]^{\sigma_{k+1}} 
& & \\
 & X_{C_{k+1}} \ar[r]^\psi \ar[d]^{\pi_{C_{k+1}}} & X_{C_k} \ar[d]^{\pi_{C_k}}\\
 & C_{k+1} \ar[r]^{\phi}  & C_k 
}
$$
Here $X_{C_{k+1}}=X_{C_k} \times_{C_k} C_{k+1}$ and 
$\sigma_{k+1}$ is the unique morphism which makes the above diagram commutative.
Let $P_{k+1} \in X$ be the rational point of $X$ corresponding to $\sigma_{k+1}$.
By (II), $\Img(\sigma_{k+1}) \not\subset (f_{C_{k+1}}^m)^{-1}(I_{f_{C_{k+1}}})$ 
for every $m \geq 0$.
Hence $\Img(f_{C_{k+1}}^m \circ \sigma_{k+1}) \not\subset I_{f_{C_{k+1}}}$ 
and so $f^m(P_{k+1}) \not\in I_f$ for every $m \geq 0$.
Therefore $P_{k+1} \in X_f$.

Let $p_{k+1}: X_{k+1} \to X_{C_{k+1}}$ be a resolution of $(X_{C_{k+1}})_{\mathrm{red}}$ 
whose exceptional locus is over a finite union of fibers of $\pi_{C_{k+1}}$ and 
$\theta=p_k^{-1} \circ \psi \circ p_{k+1}$ becomes a morphism.
Then we obtain the following diagram:
$$
\xymatrix{
X_{k+1} \ar[r]^\theta \ar[d]^{p_{k+1}} & X_k \ar[d]^{p_k} &  \\
X_{C_{k+1}} \ar[r]^\psi \ar[d]^{\pi_{C_{k+1}}} & X_{C_k} \ar[r]^{\psi_k} \ar[d]^{\pi_{C_k}} 
& X_{C_0} \ar[d]^{\pi_{C_0}} \\
C_{k+1} \ar[r]^{\phi} \ar@(ul, dl)[u]^{\sigma_{k+1}} \ar[ur]_{p_k \circ \iota} & C_k \ar[r] & C_0
}
$$
Set $f_k=p_k^{-1} \circ f_{C_k} \circ p_k$, 
$f_{k+1}=p_{k+1}^{-1} \circ f_{C_{k+1}} \circ p_{k+1}$ and 
$\sigma'_{k+1}= p_{k+1}^{-1} \circ \sigma_{k+1}$.
Fix a positive integer $m$.
Then it follows that $p_k \circ \theta \circ f_{k+1}^m \circ \sigma_{k+1}' 
= p_k \circ f_k^m \circ \iota$.
Since $p_k$ is birational, we have 
$\theta \circ f_{k+1}^m \circ \sigma_{k+1}' =f_k^m \circ \iota$.
Take an ample divisor $A$ on $X_{k+1}$ such that $A-\theta^*H$ is ample.
We compute
\begin{align*}
\deg (f_{k+1}^m \circ \sigma_{k+1}')^*A 
&\geq \deg (f_{k+1}^m \circ \sigma_{k+1}')^* \theta^*H \\
&= \deg (\theta \circ f_{k+1}^m \circ \sigma_{k+1}')^*H \\
&= \deg (f_k^m \circ \iota)^*H.
\end{align*}
By (III) and Lemma \ref{lem4.2} (ii), we have 
$$\deg (f_k^m \circ \iota)^*H=\deg \iota^*(f_k^m)^*H= ((f_k^m)^*H \cdot H^{n-1}).$$
Now $(X_{k+1} \xrightarrow{\pi_{C_{k+1}} \circ p_{k+1}} C_{k+1}, f_{k+1})$ is a 
model of $(X, f)$ and $\sigma_{k+1}'$ is a section of $\pi_{C_{k+1}} \circ p_{k+1}$ 
corresponding to $P_{k+1}$.
Therefore 
\begin{align*}
\underline \alpha_f(P_{k+1}) 
&= \liminf_{m \to \infty} \deg((f_{k+1}^m \circ \sigma_{k+1}')^*A)^{1/m}\ \ \ 
\mathrm{(by\ Proposition\ \ref{prop2.12.3})} \\
&\geq \liminf_{m \to \infty} ((f_k^m)^*H \cdot H^{n-1})^{1/m} \\
&=\delta_{f_k} = \delta_f.
\end{align*}
So $\alpha_f(P_{k+1})$ exists and $\alpha_f(P_{k+1})=\delta_f$.

Fix $i \in \{0, \ldots, k \}$  and $m, m' \geq 0$.
By (IV), $\Img(f_k^{m'} \circ \iota) \neq \Img(f_k^m \circ p_k^{-1} \circ (\sigma_i)_{C_k})
=\Img(f_k^m \circ p_k^{-1} \circ (\sigma_i)_{C_k} \circ \phi)$.
Since $p_k$ is birational and the images of both $f_k^{m'} \circ \iota$ and 
$f_k^m \circ p_k^{-1} \circ (\sigma_i)_{C_k} \circ \phi$ intersects 
the isomorphic locus of $p_k$, 
we have 
$$\Img(p_k \circ f_k^{m'} \circ \iota) 
\neq \Img(p_k \circ f_k^m \circ p_k^{-1} \circ (\sigma_i)_{C_k} \circ \phi).$$
On the other hand, 
$$p_k \circ f_k^{m'} \circ \iota = \psi \circ f_{C_{k+1}}^{m'} \circ \sigma_{k+1}$$ 
and 
$$p_k \circ f_k^m \circ p_k^{-1} \circ (\sigma_i)_{C_k} \circ \phi
= \psi \circ f_{C_{k+1}}^{m} \circ (\sigma_i)_{C_{k+1}}.$$
So $\Img(\psi \circ f_{C_{k+1}}^{m'} \circ \sigma_{k+1}) \neq 
\Img(\psi \circ f_{C_{k+1}}^{m} \circ (\sigma_i)_{C_{k+1}})$, and so 
$f_{C_{k+1}}^{m'} \circ \sigma_{k+1} \neq f_{C_{k+1}}^{m} \circ (\sigma_i)_{C_{k+1}}$.
This means that $f^{m'}(P_{k+1}) \neq f^m(P_i)$.
Hence $O_f(P_{k+1}) \cap O_f(P_i) = \varnothing$.

Set $\psi_{k+1}=\psi_k \circ \psi$.
Then $\psi_{k+1} \circ \sigma_{k+1}= \psi_k \circ p_k \circ \iota$.
By (V), $a_i \in \Img(\psi_{k+1} \circ \sigma_{k+1})$.
As a consequence, $P_1, \ldots, P_{k+1}$ satisfies $(*)_{k+1}$.

Continuing this process, we obtain a subset 
$S= \{P_1, P_2, \ldots \} \subset X_f$, 
a sequence $\cdots \to C_1 \to C_0$ of finite morphisms of curves, 
and sections $\sigma_i: C_i \to X_{C_i}$ corresponding to $P_i$ for each $i \geq 1$ 
such that 
\begin{itemize}
\item
$\alpha_f(P_i)=\delta_f$ for every $i$,
\item
$O_f(P_i) \cap O_f(P_j) = \varnothing$ if $i \neq j$, 
and 
\item
$a_i \in \Img(\psi_{C_i} \circ \sigma_i)$ for every $i$.
\end{itemize}
So it is enough to show that $S$ is a Zariski dense subset of $X$.
Let $Z \subset X$ be a proper closed subset of $X$.
We take a finite cover $C \to C_0$ such that $Z$ lifts to a proper closed subset 
$Z_C \subset X_C= X_{C_0} \times_{C_0} C$.
Since $\psi_C(Z_C)$ is a proper closed subset of $X_{C_0}$, $a_i \not\in \psi_C(Z_C)$ 
for some $i$.
Take a curve $C'$ with finite morphisms $C' \to C$ and $C' \to C_i$ which makes the 
following diagram commutative: 
\[
\xymatrix{
&C \ar[rd]&\\
C' \ar[ru] \ar[rd]&&C_{0}\\
&C_i \ar[ru]&
}
\]
Set $Z_{C'} = Z_C \times_C C' \subset X_{C'}$.
Since $a_i \in \Img(\psi_{C'} \circ (\sigma_i)_{C'})$ and 
$a_i \not\in \psi_C(Z_C)=\psi_{C'}(Z_{C'})$, 
$\Img((\sigma_i)_{C'}) \not\subset Z_{C'}$.
So $P_i \not\in Z$.
Therefore $S$ is a Zariski dense subset of $X$.
\end{proof}

Theorem \ref{thm_orbit} includes the following result.

\begin{cor}\label{cor_existence}
Assume that $k$ is an uncountable algebraically closed field of characteristic 0.
Let $(X, f)$ be a dynamical system over $\overline{k(t)}$.
Then there exists a rational point $P \in X_f$ such that 
$\alpha_f(P)=\delta_f$.
\end{cor}

For a dynamical system on a rational variety, 
we can take a subset $S$ as in the statement of Theorem \ref{thm_orbit} 
over a fixed function field.

\begin{thm}\label{thm5.4}
Assume that $k$ is an uncountable algebraically closed field of characteristic 0.
Let $(X, f)$ be a dynamical system over 
$\overline{k(t)}$ such that $X$ is rational.
Then there exists a subset $S \subset X_f$ such that 
\begin{itemize}
\item
There exists a function field $K$ of a curve over $k$ and a model 
$X_K$ of $X$ over $K$ such that all points in $S$ are defined over $K$,
\item 
$\alpha_f(P)=\delta_f$ for every $P \in S$, 
\item 
$S$ is a Zariski dense subset of $X$, and 
\item
$O_f(P) \cap O_f(Q) = \varnothing$ 
if $P, Q \in S$ and $P \neq Q$.
\end{itemize} 
\end{thm}

We need a result in \cite{MSS}. 

\begin{thm}[{\cite[Theorem 3.4 (i)]{MSS}}]\label{Theorem:BirationalInvariance}
Let $f \colon X \dashrightarrow X$ and 
$g \colon Y \dashrightarrow Y$ be dominant rational self-maps 
on smooth projective varieties and $\phi \colon Y \dashrightarrow X$ a birational map 
such that $\phi \circ g= f \circ \phi$. 
Let $V \subset Y$ be  an open subset such that 
$\phi|_V : V \to \phi(V)$ is an isomorphism. 
Then $\overline \alpha_g(Q)=\overline \alpha_f(\phi(Q))$ and 
$\underline \alpha_g(Q)=\underline \alpha_f(\phi(Q))$ 
for any $Q \in Y_g \cap \phi^{-1}(X_f)$ satisfying $O_g(Q) \subset V$. 
\end{thm}

\begin{proof}[Proof of Theorem \ref{thm5.4}]
Let $\phi: Y= \mathbb P^n_{\overline{k(t)}} \dashrightarrow X$ be a birational map.
Set $g= \phi^{-1} \circ f \circ \phi$.
Take open subsets $U \subset X$ and $V \subset Y$ such that 
$\phi|_V: V \to U$ is isomorphic.
We can take a curve $C$, a model 
$(X_C \overset{\pr_C}{\to} C, f_C)$ 
(resp.~$(Y_C=\mathbb P_{k}^n \times C  \overset{\pr_C}{\to} C, g_C)$) 
of $(X, f)$ (resp.~$(Y, g)$), 
a lift $U_C$ (resp.~$V_C$) of $U$ (resp.~$V$),
and a birational map $\phi_C: Y_C \dashrightarrow X_C$ 
such that $\phi_C|_{V_C}: V_C \to U_C$ is isomorphic.

Set $Z_C=Y_C \setminus V_C$.
By Lemma \ref{lem_points}, we can take a countable subset 
$M=\{a_i=(b_i, c_i)\}_{i=1}^\infty \subset V_C$ 
such that
\begin{itemize}
\item 
$M$ is Zariski dense in $Y_C$,
\item
$a_i \not\in (g_C^m)^{-1}(I_{g_C} \cup Z_C) \cup \phi_C^{-1} (f_C^m)^{-1}(I_{f_C})$  
for every $m \geq 0$, and
\item 
$b_i \not\in \pr_{\mathbb P_{k}^n}(I_{g_C^m})$ for every $m \geq 1$ and 
$i \geq 1$.  
\end{itemize}
Note that $\pr_{\mathbb P_{k}^n}(I_{g_C^m})$  
is a proper closed subset of $\mathbb P_{k}^n$ for every $m$ 
because $\dim I_{g_C^m} \leq n-1$.   
For $m \geq 1$, let $J_m \subset \pr_{\mathbb P_{k}^n}(I_{g^m})$ be 
the closed subset over which the fibers of 
$\pr_{\mathbb P_{k}^n}|_{I_{g^m}}$ have positive dimensions. 
Then $\codim_{\mathbb P_{k}^n} J_m \geq 2$. 
We  take a point $q \in C$ such that $q \neq c_i$ for all $i$ and 
$f_C|_{F_q}: F_q \dashrightarrow F_q$ is well-defined and dominant, where 
$F_q$ denotes the fiber of $\pr_C: Y_C \to C$ over $q$.

Assume that we have sections $\tau_1,  \ldots, \tau_k$ of $\pr_C$  
corresponding to 
$Q_1, \ldots, Q_k \in Y$ 
which satisfy the condition $(*)_k$: 
\begin{itemize}
\item
$Q_i \in Y_g \cap \phi^{-1}(X_f)$ such that $O_g(Q_i) \subset V$ for $1 \leq i \leq k$,
\item
$\alpha_g(Q_i)=\delta_g$ for $1 \leq i \leq k$,
\item
$O_g(Q_i) \cap O_g(Q_j) = \varnothing$ 
if $1 \leq i, j \leq k$ and $i \neq j$, and 
\item
$a_i \in \Img(\tau_i)$ for $1 \leq i \leq k$.
\end{itemize}
We take general hyperplanes $H_1, \ldots, H_{n-1}$ of $\mathbb P_{k}^n$.  
Then we have a line  
$L=H_1 \cap \cdots \cap H_{n-1} \subset \mathbb P_{k}^n$. 
We can choose $H_1, \ldots, H_{n-1}$ as satisfying 
\begin{itemize}
\item[(I)] 
$L \not\subset \pr_{\mathbb P_{k}^n}(I_{g_C^m})$ and 
$L \cap J_m= \varnothing$ for every $m \geq 1$,   
\item[(II)]
$L \times \{q\} \not\subset 
(g_C^{m'}|_{F_q})^{-1}(\Img(g_C^m \circ \tau_i)\cap F_q)$ 
for every $m, m' \geq 0$ and $1 \leq i \leq k$, and 
\item[(III)]
$b_{k+1} \in L$.
\end{itemize}
Note that $\Img(g_C^m \circ \tau_i)\cap F_q$ is a point since $g_C^m \circ \tau_i$ is 
a section of $\pr_C$. 
By (I), $I_{g_C^m} \cap (L \times C) 
\subset Y$ is a finite set. 
Set $\bigcup_{m \geq 1} (I_{g_C^m} \cap (L \times C))
=\{ (x_j, y_j) \}_j$. 
We can construct a finite cover $\phi: C \to L$ satisfying 
\begin{itemize}
\item[(1)]
$\phi(c_{k+1})=b_{k+1}$, 
\item[(2)]
$\phi(y_j) \neq x_j$ for every $j$, and  
\item[(3)]
$(\phi(q), q) \not\in (g^{m'}|_{F_q})^{-1}(\Img(g^m \circ \tau_i)\cap F_q)$ 
for every $m, m'$ and $1 \leq i \leq k$,
\end{itemize}
by composing a fixed finite morphism $C \to L$ 
with a suitable automorphism on $L$.

Set $\tau_{k+1}: C \overset{(\phi, \mathrm{id}_C)}{\to} L \times C 
\hookrightarrow X_C$ and let
$Q_{k+1} \in \mathbb P^n(K)$ be the corresponding rational point. 
Then $a_{k+1} \in \Img(\tau_{k+1})$ by (1).
Since $a_{k+1} \not\in (g_C^m)^{-1}(I_{g_C} \cup Z_C) \cup \phi_C^{-1} (f_C^m)^{-1}(I_{f_C})$, 
$\Img(g_C^m \circ \tau_{k+1}) \not\subset I_{g_C} \cup Z_C$ 
and $\Img(f_C^m \circ \phi_C \circ \tau_{k+1}) 
 \not\in I_{f_C}$  for every $m \geq 0$.
Hence $Q_{k+1} \in Y_g \cap \phi^{-1}(X_f)$ such that $O_g(Q_{k+1}) \subset V$.
Further, $\Img(\tau_{k+1}) \cap I_{f^m}=\varnothing$ for every $m$ by (2). 
Therefore  
$\alpha_g(Q_{k+1})$ exists and 
$\alpha_g(Q_{k+1})=\delta_g$ by Theorem \ref{thm5.0.1} (i).  
Moreover, $(g_C^{m'} \circ \tau_{k+1})(q) \neq (g^m \circ \tau_i)(q)$ 
for every $m, m'$ and $1 \leq i \leq k$ by (3).  
In particular, it follows that $g_C^{m'} \circ \tau_{k+1} \neq g_C^m \circ \tau_i$ 
for every $m, m'$ and $1 \leq i \leq k$. 
This means that 
$O_g(Q_{k+1}) \cap O_g(Q_i) = \varnothing$ for $1 \leq i \leq k$. 
As a consequence, $\{ \tau_1, \ldots, \tau_k, \tau_{k+1} \}$ satisfies $(*)_{k+1}$. 

Continuing this process, we obtain morphisms $\tau_1, \tau_2, \ldots $ and a subset 
$T= \{Q_1, Q_2, \ldots \} \subset V \cap Y_g$.
Set $P_i = \phi(Q_i)$ and $S=\{ P_1, P_2, \ldots \}$.
Then $P_i$ corresponds to a section $\sigma_i=\phi_{C} \circ \tau_i$.
Since $Q_i \in Y_g \cap \phi^{-1}(X_f)$ and $O_g(Q_i) \subset V$,
$\alpha_g(Q_i)=\alpha_f(P_i)$ by Theorem \ref{Theorem:BirationalInvariance}.
So we have $\alpha_f(P_i)=\delta_f$.
For $i, j$ with $i \neq j$,
$O_g(Q_i) \cap O_g(Q_j) = \varnothing$ implies that 
$O_f(P_i) \cap O_f(P_j) = \varnothing$.
Since $p_i \in \Img(\tau_i)$ for every $i$, $\bigcup_i \Img(\tau_i)$ is Zariski dense 
in $Y_C$ and so $\bigcup_i \Img(\sigma_i)$ is Zariski dense in $X_C$. 
So $S$ is Zariski dense in $X$. 
Therefore $S$ satisfies the claim. 
\end{proof}

%%%%%%%%%%%%%%%%%%%%%%%%%%%%%%%%%%%%%%%%%%%%%%%%%%%%%%%%%%%%%%

\end{document}